\documentclass{amsart}
\usepackage{amscd,amsmath,amssymb,euscript}

\newtheorem{theorem}{Theorem}[section]
\newtheorem{proposition}[theorem]{Proposition}
\newtheorem{lemma}[theorem]{Lemma}

\newtheorem{remark}[theorem]{Remark}

\newcommand{\Ext}{\mathop\mathrm{Ext}\nolimits}
\newcommand{\Hom}{\mathop\mathrm{Hom}\nolimits}

\newcommand{\Pic}{\mathop\mathrm{Pic}\nolimits}

\newcommand{\im}{\mathop\mathrm{im}\nolimits}

\newcommand{\splcpx}{\mathop\mathrm{Splcpx}\nolimits}

\newcommand{\uet}{^{\mbox{\rm \scriptsize{\'{e}t}}}}

\begin{document}

\title{Smoothness of the moduli space of complexes of 
coherent sheaves on an abelian or a projective K3 surface}
\author{Michi-aki Inaba}
\address{{\rm Michi-aki Inaba} \\
Department of Mathematics, Kyoto University \\
Kyoto, 606-8502, Japan}
\email{inaba@math.kyoto-u.ac.jp}

\begin{abstract}
For an abelian or a projective K3 surface $X$
over an algebraically closed field $k$,
consider the moduli space $\splcpx_{X/k}\uet$
of the objects $E$ in $D^b(\mathrm{Coh}(X))$
satisfying $\Ext^{-1}_X(E,E)=0$ and $\Hom(E,E)\cong k$.
Then we can prove that $\splcpx_{X/k}\uet$
is smooth and has a symplectic structure.
\end{abstract}

\maketitle

\markboth{Michi-aki Inaba}{Smoothness of the moduli space
of complexes of coherent sheaves on an abelian or a projective K3 surface}

\section{Introduction}

It was proved by Mukai in \cite{mukai} that the moduli space
of simple sheaves on an abelian or a projective K3 surface is smooth
and has a symplectic structure.
We will generalize this result to the moduli space
of objects in the derived category of coherent
sheaves, which is introduced in \cite{inaba1}.
By [\cite{inaba2}, Theorem 4.4], the moduli space of
(semi)stable objects with respect to a strict ample sequence
in a derived category of coherent sheaves on an abelian or a projective
K3 surface gives examples of projective symplectic varieties.

In the proof of the main results,
we will use the the trace map that also played a key role
in \cite{mukai}.
More precisely, we will calculate the image
by the trace map of the obstruction class for
the deformation of complexes of coherent sheaves.
So the idea of the proof of this paper is the same
as that of \cite{mukai}.
However, the calculation of the trace map, without any preparation,
seems to be too complicated.
For this reason, we will reconsider in section 2 the definition
of the obstruction class for the deformation of vector bundles.
By virtue of this consideration (Lemma \ref{obstruction=})
in section 2, the calculation of the trace map
becomes clear and the main result can be deduced from it.

The content of this paper was originally written as an appendix
of \cite{inaba2}.
However there was a mistake in the proof of the smoothness of
$\splcpx_{X/k}\uet$.
In this paper the author corrects the mistake.

\section{Obstruction classes for the deformation of vector bundles}

First we recall the obstruction theory of the deformation
of objects in the derived category of bounded complexes
of coherent sheaves.

Let $S$ be a noetherian scheme and $X$ be
a projective scheme flat over $S$.
We fix an $S$-ample line bundle ${\mathcal O}_X(1)$ on $X$.
Let $A$ be an artinian local ring over $S$ with residue field $k=A/m$
and $I$ be an ideal of $A$ such that $mI=0$.
Take a bounded complex $E^{\bullet}$ of $A/I$-flat coherent sheaves on $X_{A/I}$.
Then there are integers $l,l'$ such that $E^i=0$ for $i<l'$ and $i>l$.
We can take a complex $V^{\bullet}=(V^i,d^i)$ of the form 
$V^i=V_i\otimes{\mathcal O}_{X_{A/I}}(-m_i)$
and a quasi-isomorphism $V^{\bullet}\to E^{\bullet}$,
where $V_i$ are free $A$ modules of finite rank,
$V_i=0$ for $i>l$ and
$1\ll m_l\ll m_{l-1}\ll\cdots\ll m_{i+1}\ll m_i\ll\cdots$.
Take lifts 
\[
 \tilde{d}^i:V_i\otimes{\mathcal O}_{X_A}(-m_i)\to 
 V_{i+1}\otimes{\mathcal O}_{X_A}(-m_{i+1})
\]
of the homomorphisms
\[
 d^i:V_i\otimes{\mathcal O}_{X_{A/I}}(-m_i)\to 
 V_{i+1}\otimes{\mathcal O}_{X_{A/I}}(-m_{i+1}).
\]
Then we obtain homomorphisms
\[
 \delta^i:=\tilde{d}^{i+1}\circ\tilde{d}^i:
 V_i\otimes{\mathcal O}_{X_A}(-m_i)\to 
 I\otimes_A V_{i+2}\otimes{\mathcal O}_{X_A}(-m_{i+2}).
\]
We put
\[
 \omega(E^{\bullet}):=[\{\delta^i\}]\in 
 H^2(\Hom(V^{\bullet},V^{\bullet}\otimes I))\cong
 \Ext^2(E^{\bullet}\otimes k,E^{\bullet}\otimes k)\otimes_k I.
\]

\begin{proposition}
 $\omega(E^{\bullet})=0$ if and only if $E^{\bullet}$ can be lifted to
 an object of $D^b(\mathrm{Coh}(X_A))$
 of finite $\mathrm{Tor}$ dimension over $A$.
\end{proposition}

(Proof is in [\cite{inaba1},Proposition 2.3].)

For a vector bundle, there is another definition of the obstruction class.
Let $F$ be a locally free sheaf of rank $r$ on $X_{A/I}$.
Take an affine open covering $\{U_{\alpha}\}$ of $X_A$
such that 
$F|_{U_{\alpha}}\cong{\mathcal O}_{U_{\alpha}\otimes {A/I}}^{\oplus r}$
for any $\alpha$.
Let $F_{\alpha}$ be a free ${\mathcal O}_{U_{\alpha}}$-module such that
$F_{\alpha}\otimes A/I \cong F|_{U_{\alpha}}$.
Take a lift 
$\varphi_{\beta \alpha}:F_{\alpha}|_{U_{\alpha\beta}}\to
F_{\beta}|_{U_{\alpha \beta}}$
of the composite 
\[
 F_{\alpha}\otimes A/I|_{U_{\alpha\beta}}\stackrel{\sim}\longrightarrow
 F|_{U_{\alpha\beta}}\stackrel{\sim}\longrightarrow
 F_{\beta}\otimes A/I|_{U_{\alpha \beta}},
\]
where $U_{\alpha\beta}:=U_{\alpha}\cap U_{\beta}$.
We put 
\[
 \theta_{\alpha\beta\gamma}:=
 \varphi_{\gamma\alpha}^{-1}\circ\varphi_{\gamma\beta}\circ
 \varphi_{\beta\alpha}-\mathrm{id}_{F_{\alpha}}:
 F_{\alpha}|_{U_{\alpha\beta\gamma}}\longrightarrow 
 I\otimes F_{\alpha}|_{U_{\alpha\beta\gamma}},
\]
where $U_{\alpha\beta\gamma}:=U_{\alpha}\cap U_{\beta}\cap U_{\gamma}$.
Then the cohomology class
\[
 o(F):=[\{\theta_{\alpha\beta\gamma}\}]\in \check{H}^2({\mathcal End}(F)\otimes I)
 \cong \Ext^2(F,F\otimes I)
\]
can be defined.
As is stated in [\cite{sga}, III, Proposition 7.1],
we have the following proposition.

\begin{proposition}
 $o(F)=0$ if and only if $F$ can be lifted to a locally free sheaf
 on $X_A$.
\end{proposition}

A vector bundle $F$ on $X_{A/I}$ can be considered as
the object of $D^b(\mathrm{Coh}(X_{A/I}))$ whose $0$-th component is $F$
and the other components are zero.
We will show that $\omega(F)$ and $o(F)$
are the same element in $\Ext^2(F,F\otimes I)$.

We take a resolution of $F$ by locally free sheaves:
\[
 \cdots\longrightarrow V^2 \stackrel{d^2}\longrightarrow V^1
 \stackrel{d^1}\longrightarrow V^0 \stackrel{\pi}\longrightarrow
 F \longrightarrow 0,
\]
where each $V^i$ is isomorphic to $V_i\otimes{\mathcal O}_{X_{A/I}}(-m_i)$
for a free $A$-module $V_i$ of finite rank and
$1\ll m_0\ll m_1\ll\cdots\ll m_i\ll m_{i+1}\ll\cdots$.
Then we have a quasi-isomorphism
${\mathcal Hom}(F,F)\otimes I \to
{\mathcal Hom}^{\bullet}(V^{\bullet},F)\otimes I$.
Let 
\[
 {\mathcal Hom}(F,F)\otimes I \to
 {\mathcal C}^{\bullet}({\mathcal Hom}(F,F)\otimes I)
\]
be the \v{C}ech resolution of ${\mathcal Hom}(F,F)\otimes I$
with respect to the covering $\{U_{\alpha}\}$
and 
\[
 {\mathcal Hom}^{\bullet}(V^{\bullet},F)\otimes I\to
 {\mathcal C}^{\bullet}({\mathcal Hom}^{\bullet}(V^{\bullet},F)\otimes I)
\]
be that of ${\mathcal Hom}^{\bullet}(V^{\bullet},F)\otimes I$.
Then we obtain a composition of isomorphisms
\[
 f:H^2(\Hom^{\bullet}(V^{\bullet},F))\stackrel{\sim}\longrightarrow
 {\bf H}^2(C^{\bullet}({\mathcal Hom}^{\bullet}(V^{\bullet},F)\otimes I))
 \stackrel{\sim}\longrightarrow
 \check{H}^2({\mathcal End}(F)\otimes I),
\]
where $C^{\bullet}({\mathcal Hom}^{\bullet}(V^{\bullet},F)\otimes I)=
\Gamma(X,{\mathcal C}^{\bullet}({\mathcal Hom}^{\bullet}(V^{\bullet},F)\otimes I))$.

\begin{lemma}\label{obstruction=}
Under the above assumption and notation,
we have $f(\omega(F))=o(F)$.
\end{lemma}

\begin{proof}
First note that the element $\omega(F)$ is defined by
\[
 \omega(F)=[\{(\pi\otimes\mathrm{id}_I)
 \circ(\tilde{d}^1\circ\tilde{d}^2)\}]
 \in H^2(\Hom^{\bullet}(V^{\bullet},F\otimes I)),
\]
where $\tilde{d}^i:V_i\otimes{\mathcal O}_{X_A}(-m_i)\to
V_{i+1}\otimes{\mathcal O}_{X_A}(-m_{i+1})$
is a lift of $d^i$.
Replacing $\{U_{\alpha}\}$ by its refinement, we may assume that
$\ker d^2|_{U_{\alpha}}$,
$\im d^2|_{U_{\alpha}}$,
$\im d^1|_{U_{\alpha}}$,
$V_2\otimes{\mathcal O}_{X_{A/I}}(-m_2)|_{U_{\alpha}}$,
$V_1\otimes{\mathcal O}_{X_{A/I}}(-m_1)|_{U_{\alpha}}$
and $F|_{U_{\alpha}}$ are all free sheaves.
Then the exact sequences
\begin{gather*}
 0 \longrightarrow \ker d^2|_{U_{\alpha}}
 \stackrel{i_2}\longrightarrow
 V_2\otimes{\mathcal O}_{X_{A/I}}(-m_2)|_{U_{\alpha}}
 \stackrel{p_2}\longrightarrow
 \im d^2|_{U_{\alpha}} \longrightarrow 0, \\
 0 \longrightarrow \im d^2|_{U_{\alpha}}
 \stackrel{i_1}\longrightarrow
 V_1\otimes{\mathcal O}_{X_{A/I}}(-m_1)|_{U_{\alpha}}
 \stackrel{p_1}\longrightarrow
 \im d^1|_{U_{\alpha}} \longrightarrow 0, \\
 0 \longrightarrow \im d^1|_{U_{\alpha}}
 \stackrel{i_0}\longrightarrow
 V_0\otimes{\mathcal O}_{X_{A/I}}(-m_0)|_{U_{\alpha}}
 \xrightarrow{\pi|_{U_{\alpha}}}
 F|_{U_{\alpha}} \longrightarrow 0 
\end{gather*}
split and we can take free ${\mathcal O}_{U_{\alpha}}$-modules
$F_{\alpha}$, $I_1^{\alpha}$, $I_2^{\alpha}$ such that
$F_{\alpha}\otimes A/I\cong F|_{U_{\alpha}}$
and $I_i^{\alpha}\otimes A/I\cong\im d^i|_{U_{\alpha}}$ for $i=1,2$.
Taking lifts $\tilde{i}_0^{\alpha}$, $\tilde{i}_1^{\alpha}$, $\pi_{\alpha}$,
$\tilde{p}_1^{\alpha}$, $\tilde{p}_2^{\alpha}$ of
$i_0$, $i_1$, $\pi|_{U_{\alpha}}$, $p_1$, $p_2$,
we obtain splitting exact sequences
\begin{gather*}
 0 \longrightarrow \ker \tilde{p}_2^{\alpha} \longrightarrow
 V_2\otimes{\mathcal O}_{X_{A}}(-m_2)|_{U_{\alpha}}
 \stackrel{\tilde{p}_2^{\alpha}}\longrightarrow
 I_2^{\alpha} \longrightarrow 0, \\
 0 \longrightarrow I_2^{\alpha} \stackrel{\tilde{i}_1^{\alpha}}\longrightarrow
 V_1\otimes{\mathcal O}_{X_{A}}(-m_1)|_{U_{\alpha}}
 \stackrel{\tilde{p}_1^{\alpha}}\longrightarrow I_1^{\alpha} \longrightarrow 0, \\
 0 \longrightarrow I_1^{\alpha} \stackrel{\tilde{i}_0^{\alpha}}\longrightarrow
 V_0\otimes{\mathcal O}_{X_{A}}(-m_0)|_{U_{\alpha}}
 \stackrel{\pi_{\alpha}}\longrightarrow F_{\alpha}\longrightarrow 0.
\end{gather*}
Let
\begin{gather*}
 \tilde{s}_2^{\alpha}:I_2^{\alpha}\longrightarrow
 V_2\otimes{\mathcal O}_{X_{A}}(-m_2)|_{U_{\alpha}}, \\
 \tilde{r}_1^{\alpha}:V_1\otimes{\mathcal O}_{X_{A}}(-m_1)|_{U_{\alpha}}
 \longrightarrow I_2^{\alpha}, \quad
 \tilde{s}_1^{\alpha}:I_1^{\alpha}\longrightarrow
 V_1\otimes{\mathcal O}_{X_{A}}(-m_1)|_{U_{\alpha}}, \\
 \tilde{r}_0^{\alpha}:V_0\otimes{\mathcal O}_{X_{A}}(-m_0)|_{U_{\alpha}}
 \longrightarrow I_1^{\alpha}, \quad
 \nu_{\alpha}:F_{\alpha}\longrightarrow
 V_0\otimes{\mathcal O}_{X_{A}}(-m_0)|_{U_{\alpha}}
\end{gather*}
be splittings.
Put
\begin{gather*}
 d^2_{\alpha}:V_2\otimes{\mathcal O}_{X_A}(-m_2)|_{U_{\alpha}}
 \stackrel{\tilde{p}_2^{\alpha}}\longrightarrow I_2^{\alpha}
 \stackrel{\tilde{i}_1^{\alpha}}\longrightarrow
 V_1\otimes{\mathcal O}_{X_A}(-m_1)|_{U_{\alpha}}, \\
 d^1_{\alpha}:V_1\otimes{\mathcal O}_{X_A}(-m_1)|_{U_{\alpha}}
 \stackrel{\tilde{p}_1^{\alpha}}\longrightarrow I_1^{\alpha}
 \stackrel{\tilde{i}_0^{\alpha}}\longrightarrow
 V_0\otimes{\mathcal O}_{X_A}(-m_0)|_{U_{\alpha}}, \\
 \tau_{\alpha}:V_0\otimes{\mathcal O}_{X_A}(-m_0)|_{U_{\alpha}}
 \stackrel{\tilde{r}_0^{\alpha}}\longrightarrow I_1^{\alpha}
 \stackrel{\tilde{s}_1^{\alpha}}\longrightarrow
 V_1\otimes{\mathcal O}_{X_A}(-m_1)|_{U_{\alpha}}, \\
 \sigma_{\alpha}:V_1\otimes{\mathcal O}_{X_A}(-m_1)|_{U_{\alpha}}
 \stackrel{\tilde{r}_1^{\alpha}}\longrightarrow I_2^{\alpha}
 \stackrel{\tilde{s}_2^{\alpha}}\longrightarrow
 V_2\otimes{\mathcal O}_{X_A}(-m_2)|_{U_{\alpha}}.
\end{gather*}
We consider the following diagram:
\[
 \begin{array}{ccccc}
  \Hom(V^0,F\otimes I) & \longrightarrow & \Hom(V^1,F\otimes I) &
  \longrightarrow & \Hom(V^2,F\otimes I) \\
  \downarrow & & \downarrow & & \downarrow \\
  C^0({\mathcal Hom}(V^0,F\otimes I)) & \longrightarrow &
  C^0({\mathcal Hom}(V^1,F\otimes I)) & \longrightarrow &
  C^0({\mathcal Hom}(V^2,F\otimes I)) \\
  \downarrow & & \downarrow & & \downarrow \\
  C^1({\mathcal Hom}(V^0,F\otimes I)) & \longrightarrow &
  C^1({\mathcal Hom}(V^1,F\otimes I)) & \longrightarrow &
  C^1({\mathcal Hom}(V^2,F\otimes I)) \\
  \downarrow & & \downarrow & & \downarrow \\
  C^2({\mathcal Hom}(V^0,F\otimes I)) & \longrightarrow &
  C^2({\mathcal Hom}(V^1,F\otimes I)) & \longrightarrow &
  \, C^2({\mathcal Hom}(V^2,F\otimes I)),
 \end{array}
\]
where we put $V^i:=V_i\otimes{\mathcal O}_{X_A}(-m_i)$ for $i=0,1,2$.
The image of $\omega(F)$ in
${\bf H}^2(C^{\bullet}({\mathcal Hom}^{\bullet}(V^{\bullet},F)\otimes I))$
can be represented by
\[
 \left\{ (\pi\otimes\mathrm{id}_I)\circ
 \tilde{d}^1\circ\tilde{d}^2|_{U_{\alpha}}\right\}
 \in C^0({\mathcal Hom}(V^2,F)\otimes I),
\]
which defines the same element in
${\bf H}^2(C^{\bullet}({\mathcal Hom}^{\bullet}(V^{\bullet},F)\otimes I))$ as
\[
 \left\{ (\pi\otimes\mathrm{id}_I)\circ\tilde{d}^1\circ
 \tilde{d}^2\circ(\sigma_{\alpha}-\sigma_{\beta})\right\}
 \in C^1({\mathcal Hom}(V^1,F)\otimes I).
\]
On the other hand, the image of the element
\[
 \left\{ (\pi\otimes\mathrm{id}_I)\circ
 \left(d^1_{\alpha}-\tilde{d}^1\circ
 (1-\tilde{d}^2\sigma_{\alpha})\right)\right\}
 \in C^0({\mathcal Hom}(V^1,F)\otimes I)
\]
by the homomorphism
$C^0({\mathcal Hom}(V^1,F)\otimes I)\rightarrow
C^0({\mathcal Hom}(V^2,F)\otimes I)$ is
\begin{align*}
 \left\{ (\pi \otimes \mathrm{id}_I)\circ\left( d^1_{\alpha}-\tilde{d}^1
 \circ(1-\tilde{d}^2\circ\sigma_{\alpha} ) \right)\circ d^2_{\alpha}\right\}
 &= \left\{ (\pi\otimes\mathrm{id}_I)(d^1_{\alpha}\circ d^2_{\alpha}-\tilde{d}^1\circ d^2_{\alpha}
 +\tilde{d}^1\circ\tilde{d}^2\circ\sigma_{\alpha}\circ d^2_{\alpha}) \right\}  \\
 &= \left\{ (\pi \otimes \mathrm{id}_I)\left(-\tilde{d}^1\circ d^2_{\alpha}
 +\tilde{d}^1\circ\tilde{d}^2\circ\sigma_{\alpha}\circ d^2_{\alpha}\right)\right\}  \\
 &=\left\{ (\pi\otimes\mathrm{id}_I)\left( -\tilde{d}^1\circ d^2_{\alpha} +
 \tilde{d}^1\circ \tilde{d}^2\circ \tilde{s}_2^{\alpha}\circ\tilde{r}_1^{\alpha}\circ\tilde{i}_1^{\alpha}\circ\tilde{p}_2^{\alpha}
 \right) \right\}  \\
 &=\left\{ (\pi\otimes\mathrm{id}_I)\left(-\tilde{d}^1\circ d^2_{\alpha}\circ\tilde{s}^{\alpha}_2\circ\tilde{p}_2^{\alpha}
 +\tilde{d}^1\circ \tilde{d}^2\circ\tilde{s}^{\alpha}_2\circ \tilde{p}_2^{\alpha} \right) \right\}  \\
 &=\left\{ (\pi\otimes\mathrm{id}_I)\circ \tilde{d}^1\circ (\tilde{d}^2-d^2_{\alpha})\circ\tilde{s}_2^{\alpha}\circ\tilde{p}_2^{\alpha}
 \right\}  \\
 &=0.
\end{align*}
Since
\begin{align*}
 &\left\{ (\pi\otimes\mathrm{id}_I)\circ\tilde{d}^1
 \circ\tilde{d}^2\circ(\sigma_{\alpha}-\sigma_{\beta}) \right\} 
 + d \left\{ (\pi\otimes\mathrm{id}_I)\circ
 \left(d^1_{\alpha}-\tilde{d}^1\circ
 (1-\tilde{d}^2\circ\sigma_{\alpha})\right) \right\} \\
 &=\left\{(\pi\otimes\mathrm{id}_I)\circ\tilde{d}^1\circ\tilde{d}^2
 \circ(\sigma_{\alpha}-\sigma_{\beta})\right\}
 +\left\{ (\pi\otimes\mathrm{id}_I)\circ
 \left(d^1_{\beta}-\tilde{d}^1\circ(1-\tilde{d}^2\circ\sigma_{\beta})
 \right)|_{U_{\alpha}\cap U_{\beta}}\right\} \\
 &\quad -\left\{(\pi\otimes\mathrm{id}_I)\circ
 \left(d^1_{\alpha}-\tilde{d}^1\circ(1-\tilde{d}^2\circ\sigma_{\alpha})
 \right)|_{U_{\alpha}\cap U_{\beta}}\right\}  \\
 &=-\left\{ (\pi\otimes\mathrm{id}_I)\circ
 (d^1_{\alpha}-d^1_{\beta}) \right\},
\end{align*}
we can see that
$\left\{ (\pi\otimes\mathrm{id}_I)\circ\tilde{d}^1\circ\tilde{d}^2(\sigma_{\alpha}-\sigma_{\beta}) \right\}$
and
$-\left\{ (\pi\otimes \mathrm{id}_I)\circ(d^1_{\alpha}-d^1_{\beta}) \right\}$
define the same element in
$\mathbf{H}^2(C^{\bullet}({\mathcal Hom}^{\bullet}(V^{\bullet},F)\otimes I))$.
We can see that the element
$-\{(\pi\otimes\mathrm{id}_I)\circ(d^1_{\alpha}-d^1_{\beta})\}$
defines the same element as
\begin{align*}
 & -\left\{ (\pi\otimes\mathrm{id}_I)\circ
 \left( (d^1_{\beta}-d^1_{\gamma})\circ\tau_{\beta}
 -(d^1_{\alpha}-d^1_{\gamma})\circ\tau_{\alpha}
 +(d^1_{\alpha}-d^1_{\beta})\circ\tau_{\alpha} \right) \right\} \\
 &=\left\{ (\pi\otimes\mathrm{id}_I)\circ (d^1_{\beta}-d^1_{\gamma})
 \circ (\tau_{\alpha}-\tau_{\beta}) \right\}
 \in C^2({\mathcal Hom}(V^0,F)\otimes I)
\end{align*}
in ${\bf H}^2(C^{\bullet}({\mathcal Hom}^{\bullet}(V^{\bullet},F)\otimes I))$.
Thus $\omega(F)$ is equal to the element given by
\[
 \left\{ (\pi\otimes\mathrm{id}_I)\circ(d^1_{\beta}-d^1_{\gamma})
 \circ (\tau_{\alpha}-\tau_{\beta}) \right\}
 \in C^2({\mathcal Hom}(V^0,F)\otimes I)
\]
in ${\bf H}^2(C^{\bullet}({\mathcal Hom}^{\bullet}(V^{\bullet},F)\otimes I))$.
On the other hand, the element $o(F)$ is given by
\[
 \{ (\pi_{\gamma}\circ\nu_{\alpha})^{-1}
 \circ\pi_{\gamma}\circ\nu_{\beta}
 \circ\pi_{\beta}\circ\nu_{\alpha}-\mathrm{id}_{F_{\alpha}} \} 
\]
in $\check{H}^2({\mathcal End}(F)\otimes I)$,
whose image in
${\bf H}^2(C^{\bullet}({\mathcal Hom}^{\bullet}(V^{\bullet},F)\otimes I))$
is represented by
\begin{align*}
 &\{ (\pi_{\gamma}\circ\nu_{\alpha})^{-1}
 \circ\pi_{\gamma}\circ\nu_{\beta}\circ\pi_{\beta}
 \circ\nu_{\alpha}\circ\pi_{\alpha}-\pi_{\alpha} \} \\
 & = \{ (\pi_{\gamma}\circ\nu_{\alpha})^{-1}\circ
 (\pi_{\gamma}\circ\nu_{\beta}\circ\pi_{\beta}
 \circ\nu_{\alpha}\circ\pi_{\alpha}
 -\pi_{\gamma}\circ\nu_{\alpha}\circ\pi_{\alpha}) \} \\
 &= \{ (\pi_{\gamma}\circ\nu_{\alpha})^{-1}\circ\pi_{\gamma}
 \circ(\nu_{\beta}\circ\pi_{\beta}-1)\circ
 \nu_{\alpha}\circ\pi_{\alpha} \} \\
 &= \{(\pi_{\gamma}\circ\nu_{\alpha})^{-1}\circ\pi_{\gamma}
 \circ(-d^1_{\beta}\circ\tau_{\beta})\circ(1-d^1_{\alpha}
 \circ\tau_{\alpha}) \} \\
 &= \left\{(\pi_{\gamma}\circ\nu_{\alpha})^{-1}\circ
 \left( \pi_{\gamma}\circ d^1_{\beta}\circ
 (\tau_{\alpha}-\tau_{\beta})
 -\pi_{\gamma}\circ d^1_{\beta}\circ(\tau_{\alpha}-\tau_{\beta})
 \circ d^1_{\alpha}\circ\tau_{\alpha}\right) \right\}. \\
 \end{align*}
 Here we have
 \begin{align*}
  &\pi_{\gamma}\circ d^1_{\beta}\circ(\tau_{\alpha}-\tau_{\beta})\circ d^1_{\alpha}\circ\tau_{\alpha}  \\
  &=\pi_{\gamma}\circ d^1_{\beta}\circ(\tilde{s}_1^{\alpha}\circ\tilde{r}^{\alpha}_0-\tilde{s}_1^{\beta}\circ\tilde{r}_0^{\beta})
  \circ d^1_{\alpha}\circ\tau_{\alpha}   \\
  &=\pi_{\gamma}\circ d^1_{\beta}\circ\tilde{s}^{\alpha}_1\circ\tilde{r}_0^{\alpha}\circ d^1_{\alpha}\circ\tau_{\alpha}
  -\pi_{\gamma}\circ d^1_{\beta}\circ \tilde{s}^{\beta}_1\circ\tilde{r}_0^{\beta}\circ(d^1_{\alpha}-d^1_{\beta})\circ\tau_{\alpha}
  -\pi_{\gamma}\circ d^1_{\beta}\circ\tilde{s}_1^{\beta}\circ\tilde{r}_0^{\beta}\circ d^1_{\beta}\circ\tau_{\alpha}  \\
  &=\pi_{\gamma}\circ d^1_{\beta}\circ\tilde{s}_1^{\alpha}\circ\tilde{r}_0^{\alpha}\circ\tilde{i}_0^{\alpha}\circ\tilde{p}_1^{\alpha}\circ\tau_{\alpha}
  -\pi_{\gamma}\circ d^1_{\beta}\circ\tilde{s}_1^{\beta}\circ\tilde{r}_0^{\beta}\circ\tilde{i}_0^{\beta}\circ\tilde{p}_1^{\beta}\circ\tau_{\alpha} \\
  &=\pi_{\gamma}\circ d^1_{\beta}\circ\tilde{s}_1^{\alpha}\circ\tilde{p}_1^{\alpha}\circ\tau_{\alpha}
  -\pi_{\gamma}\circ d^1_{\beta}\circ\tilde{s}_1^{\beta}\circ\tilde{p}_1^{\beta}\circ\tau_{\alpha}   \\
  &=\pi_{\gamma}\circ d^1_{\beta}\circ(\mathrm{id}-\tilde{i}_1^{\alpha}\circ\tilde{r}_1^{\alpha})\circ\tau_{\alpha}
  -\pi_{\gamma}\circ d^1_{\beta}\circ(\mathrm{id}-\tilde{i}_1^{\beta}\circ\tilde{r}_1^{\beta})\circ\tau_{\alpha}  \\
  &=\pi_{\gamma}\circ d^1_{\beta}\circ\tilde{i}_1^{\alpha}\circ\tilde{r}_1^{\alpha}\circ\tau_{\alpha}
  -\pi_{\gamma}\circ d^1_{\beta}\circ\tilde{i}_1^{\beta}\circ\tilde{r}_1^{\beta}\circ\tau_{\alpha}  \\
  &=\pi_{\gamma}\circ d^1_{\beta}\circ\tilde{i}_1^{\alpha}\circ\tilde{r}_1^{\alpha}\circ\tilde{s}_1^{\alpha}\circ\tilde{r}_0^{\alpha}
  \quad \text{(note that $d^1_{\beta}\circ\tilde{i}_1^{\beta}=0$)} \\
  &=0. \quad (\text{note that $\tilde{r}_1^{\alpha}\circ\tilde{s}_1^{\alpha}=0$})
 \end{align*}
So the image of $o(F)$ in $\mathbf{H}^2(C^{\bullet}({\mathcal Hom}^{\bullet}(V^{\bullet},F)\otimes I))$ is
\begin{align*}
 \left\{(\pi_{\gamma}\circ\nu_{\alpha})^{-1}\circ\pi_{\gamma}
 \circ d^1_{\beta}\circ(\tau_{\alpha}-\tau_{\beta}) \right\}
 &=\left\{(\pi_{\gamma}\circ\nu_{\alpha})^{-1}\circ
 \pi_{\gamma}\circ (d^1_{\beta}-d^1_{\gamma})\circ
 (\tau_{\alpha}-\tau_{\beta}) \right\} \\
 &=\left\{(\pi\otimes\mathrm{id}_I)\circ (d^1_{\beta}-d^1_{\gamma})
 \circ(\tau_{\alpha}-\tau_{\beta}) \right\} 
\end{align*}
Thus we have the equality
$f(\omega(F))=o(F)$.
\end{proof}

\begin{remark}\rm
Several authors introduced obstruction classes for the deformation
of vector bundles and coherent sheaves.
For example, [\cite{H-L}, Chap 2, Appendix]
is a good reference.
However, it is not so clear that these definitions
are all equivalent.
\end{remark}

\section{Smoothness and symplectic structure}

Let $X$ be a projective scheme over a noetherian scheme $S$,
which is flat over $S$.
We define a functor $\mathrm{Splcpx}_{X/S}$ of the category 
of locally noetherian schemes to that of sets by putting
\[
 \mathrm{Splcpx}_{X/S}(T):= \left\{ E^{\bullet} \left|
 \begin{array}{l}
 \mbox{$E^{\bullet}$ is a bounded complex of $T$-flat coherent} \\
 \mbox{${\mathcal O}_{X_T}$-modules such that for any $t\in T$,} \\
 \mbox{$E^{\bullet}(t)$ satisfies the following condition $(*)$}
 \end{array}
 \right\} \right/\sim,
\]
where $T$ is a locally noetherian scheme over $S$ and
$E^{\bullet}\sim F^{\bullet}$ if there is a line bundle $L$ on $T$ such that
$E^{\bullet}\cong F^{\bullet}\otimes L$ in $D(X_T)$.
Here $D(X_T)$ is the derived category of
${\mathcal O}_{X_T}$-modules and the condition $(*)$ is
\[
 (*)\quad \Ext^i(E^{\bullet}(t),E^{\bullet}(t))\cong
 \begin{cases}
  0 & \text{if $i=-1$} \\
  k(t) & \text{if $i=0$}.
 \end{cases}
\]
Note that we denote $E^{\bullet}\otimes^{\mathbf{L}}k(t)$ by $E^{\bullet}(t)$.
Let $\mathrm{Splcpx}\uet_{X/S}$ be the \'{e}tale sheafification of
$\mathrm{Splcpx}_{X/S}$.

\begin{theorem}
 $\mathrm{Splcpx}\uet_{X/S}$ is represented by an algebraic space
 over $S$.
\end{theorem}

(Proof is in [\cite{inaba1}, Theorem 0.2].
This result was generalized by Lieblich in \cite{Lieblich}
for $X$ proper over $S$.)

\begin{theorem}
 If $X$ is an abelian or a projective K3 surface over an algebraically
 closed field $k$,
 $\mathrm{Splcpx}\uet_{X/k}$ is smooth over $k$.
\end{theorem}

\begin{proof}
Take an artinian local ring $A$ over $k$ with residue field $k=A/m$
and an ideal $I$ of $A$ such that $mI=0$.
It is sufficient to show that
$\mathrm{Splcpx}_{X/k}(A)\to\mathrm{Splcpx}_{X/k}(A/I)$ is surjective.
Indeed we can take a scheme $U$ locally of finite type over $k$
and a morphism $p:U\rightarrow\splcpx_{X/k}$ such that
the composite
$U\stackrel{p}\rightarrow\splcpx_{X/k}\stackrel{\iota}\rightarrow
\splcpx_{X/k}\uet$
is \'etale and surjective.
Take any artinian local ring $A$ over $k$ with residue field $k=A/m$
and an ideal $I$ of $A$ such that $mI=0$.
Take any member $x\in U(A/I)$.
By the surjectivity of $\splcpx_{X/k}(A)\rightarrow\splcpx_{X/k}(A/I)$,
we can take an element $y\in\splcpx_{X/k}(A)$ such that
$y\otimes A/I=p(x)$.
Then $\iota(y)\in\splcpx_{X/k}\uet(A)$
and $\iota(y)\otimes A/I=(\iota\circ p)(x)$.
Since $\iota\circ p:U\rightarrow\splcpx_{X/k}\uet$ is \'etale,
there is an element $z\in U(A)$ such that
$z\otimes A/I=x$ and $(\iota\circ p)(z)=y$.
Thus $U$ is smooth over $k$.

Let $E^{\bullet}$ be an $A/I$-valued point of $\mathrm{Splcpx}_{X/k}$.
Put $E^{\bullet}_0:=E^{\bullet}\otimes k$
and
\[
 l':=\min \{ i | \text{$H^i( E^{\bullet}_0\otimes^{\mathbf{L}} k(x))\neq 0$ for some $x\in X$}\}.
\]
We may assume that $E^{\bullet}$ is of the form
\[
 \cdots \longrightarrow 0 \longrightarrow 0 \longrightarrow E^{l'}
 \stackrel{d^{l'}_{E^{\bullet}}}\longrightarrow V^{l'+1} \stackrel{d^{l'+1}}\longrightarrow \cdots \longrightarrow V^l
 \stackrel{d^l}\longrightarrow 0 \longrightarrow 0 \cdots,
\]
where $E^{l'}$ is a vector bundle on $X_{A/I}$,
$V^i=V_i\otimes{\mathcal O}_{X_{A/I}}(-m_i)$
with $V_i$ a finite dimensional vector space over $k$,
${\mathcal O}_X(1)$ a fixed ample line bundle on $X$ and
$1\ll m_l \ll m_{l-1} \ll \cdots \ll m_{l'+1}$.
We can see that
$d^{l'}_{E^{\bullet}_0}\otimes k(x)$ is not injective for some $x\in X$.
Take a resolution
\[
 \cdots\longrightarrow V_i\otimes{\mathcal O}_{X_{A/I}}(-m_i)\longrightarrow\cdots
 \longrightarrow V_{l'}\otimes{\mathcal O}_{X_{A/I}}(-m_{l'})
 \stackrel{\pi}\longrightarrow E^{l'}
 \longrightarrow 0,
\]
where each $V_i$ is a vector space over $k$ of finite dimension and 
\[
 m_{l'+1}\ll m_{l'}\ll \cdots \ll m_i\ll m_{i-1} \ll \cdots.
\]
We put $V^i=V_i\otimes{\mathcal O}_{X_{A/I}}(-m_i)$
for $i\leq l$ and $V^i=0$ for $i>l$.
Let $V^{\bullet}$ be the complex
\[
 \cdots\longrightarrow V^i\longrightarrow V^{i+1}
 \longrightarrow\cdots\longrightarrow V^{l'}
 \xrightarrow{d^{l'}_{E^{\bullet}}\circ\pi}
 V^{l'+1}\longrightarrow\cdots\longrightarrow
 V^l\longrightarrow 0\longrightarrow\cdots.
\]
Then there is a canonical quasi-isomorphism
\[
 V^{\bullet} \longrightarrow E^{\bullet}.
\]
Put $V^{\bullet}_0:=V^{\bullet}\otimes k$.
Let
\[
 \mathrm{tr}^{\bullet}:{\mathcal Hom}^{\bullet}(E^{\bullet}_0,E^{\bullet}_0)
 \stackrel{\sim}\longrightarrow
 {\mathcal Hom}^{\bullet}({\mathcal Hom}
 ^{\bullet}(E^{\bullet}_0,E^{\bullet}_0),{\mathcal O}_X)
 \longrightarrow {\mathcal O}_X
\]
be the dual of the canonical morphism
\[
 {\mathcal O}_X \longrightarrow
 {\mathcal Hom}^{\bullet}(E^{\bullet}_0,E^{\bullet}_0) ;
 \quad 1\mapsto \mathrm{id}_{E^{\bullet}_0}.
\]
Note that
$\mathrm{tr}^p=0$ on ${\mathcal Hom}^p(E_0^{\bullet},E_0^{\bullet})$ for $p\neq 0$
and $\mathrm{tr}^0(\{x^i\})=\sum_i (-1)^i \mathrm{tr}(x^i)$
for $x^i \in {\mathcal Hom}(E^i_0,E^i_0)$.
$\mathrm{tr}^{\bullet}$ is also introduced in [\cite{H-L}, Chapter 10].
There is a commutative diagram
\[
 \begin{CD}
 \Ext^2_X(E^{\bullet}_0,E^{\bullet}_0) @>H^2(\mathrm{tr}^{\bullet})>>
  H^2(X,{\mathcal O}_X) \\
 @V s_1 V\cong V  @V s_2 V\cong V \\
 \Hom_{D(X)}(E^{\bullet}_0,E^{\bullet}_0)^{\vee} @>>>
  H^0(X,{\mathcal O}_X)^{\vee},
\end{CD}
\]
where $s_1,s_2$ are the isomorphisms determined  by
Grothendieck-Serre duality and the bottom row is the dual of
$k=H^0({\mathcal O}_X)\to \Hom_{D(X)}(E^{\bullet}_0,E^{\bullet}_0)$,
which is bijective since $E^{\bullet}_0$ is simple.
Thus the homomorphism
\[
 \Ext^2_X(E^{\bullet}_0,E^{\bullet}_0)
 \xrightarrow{H^2(\mathrm{tr}^{\bullet})} H^2(X,{\mathcal O}_X) 
\]
is an isomorphism.

Note that there is a commutative diagram
\[
 \begin{CD}
  {\mathcal Hom}^{\bullet}(E^{l'}[-l'],I\otimes E^{\bullet})
  @>>> {\mathcal Hom}(E^{l'},I\otimes E^{l'}) \\ 
  @VVV   @VV (-1)^{l'}\mathrm{tr} V \\
  {\mathcal Hom}^{\bullet}(E^{\bullet},I\otimes E^{\bullet})
  @>\mathrm{tr}>> {\mathcal O}_X\otimes I.
 \end{CD}
\]
From the above commutative diagram, we obtain a commutative diagram
\[
(\dag\dag) \quad
 \begin{CD}
  \Ext^2(E^{l'}[-l'],I\otimes E^{\bullet}) @>\tau>> \Ext^2(E^{l'},I\otimes E^{l'}) \\
  @V\sigma VV   @V (-1)^{l'}H^2(\mathrm{tr}) VV   \\
  \Ext^2(E^{\bullet},I\otimes E^{\bullet}) @>H^2(\mathrm{tr})>> H^2({\mathcal O}_X)\otimes I.
 \end{CD}
\]
Note that the morphism
\[
 \Hom(E^{\bullet}_0,E^{\bullet}_0)\longrightarrow \Hom(E^{\bullet}_0,E^{l'}_0[-l'])
\]
is not zero, since the image of $\mathrm{id}$ by this morphism
is the canonical morphism $\iota:E^{\bullet}_0\rightarrow E^{l'}_0[-l']$
which is not zero because
\[
 H^{l'}(\iota\otimes k(x)):\ker(d^{l'}_{E^{\bullet}_0}\otimes k(x))= H^{l'}(E^{\bullet}_0\otimes k(x))\longrightarrow
 H^l(E^{l'}_0[-l']\otimes k(x))=E^{l'}_0\otimes k(x)
\]
is not zero.
By Grothendieck-Serre duality, we can see that
\[
 \iota^*:\Ext^2(E^{l'}_0[-l'],E^{\bullet}_0)\longrightarrow \Ext^2(E^{\bullet}_0,E^{\bullet}_0)
\]
is not zero.
Since $\Ext^2(E^{\bullet}_0,E^{\bullet}_0)\cong k$, $\iota^*$ is surjective.
So the morphism
\[
 \sigma: \Ext^2(E^{l'}[-l'],I\otimes E^{\bullet})\longrightarrow
 \Ext^2(E^{\bullet},I\otimes E^{\bullet})
\]
is also surjective.

Take an obstruction class $\omega(E^{\bullet})\in\Ext^2(E^{\bullet},I\otimes E^{\bullet})$
for the lifting of $E^{\bullet}$ to an $A$-valued point of $\splcpx_{X/k}$.
Then there is a member $\varphi=[(\varphi^i)]\in \Ext^2(E^{l'}[-l'], I\otimes E^{\bullet})$
such that
$\sigma(\varphi)=\omega(E^{\bullet})$.
Here $\varphi^i:V^i\rightarrow I\otimes E^{i+2}$ ($i\leq l'$)
and $\varphi^i=0$ for $i>l'$.
There is an element $\gamma=(\gamma^i)\in \Hom^1(V^{\bullet},I\otimes E^{\bullet})$ such that
\begin{gather*}
 \gamma^{i+1}\circ d^i_{V^{\bullet}}+d^{i+1}_{E^{\bullet}}\circ\gamma^i=\tilde{d}^{i+1}\circ\tilde{d}^i-\varphi^i
 \quad (\text{for $i\geq l'-1$}) \\
 \gamma^{l'-1}\circ d^{l'-2}_{V^{\bullet}}=\pi\circ\tilde{d}^{l'-1}\circ\tilde{d}^{l'-2}-\varphi^{l'-2},
\end{gather*}
where $\tilde{d}^i:V_i\otimes{\mathcal O}_{X_A}(-m_i)\rightarrow V_{i+1}\otimes{\mathcal O}_{X_A}(-m_{i+1})$
is a lift of $d^i_{V^{\bullet}}$.
We can see that the image of $\varphi$ by the morphism
$\tau:\Ext^2(E^{l'}[-l'],I\otimes E^{\bullet})\rightarrow \Ext^2(E^{l'},I\otimes E^{l'})$
is given by
$[\pi\circ\tilde{d}^{l'-1}\circ\tilde{d}^{l'-2}]$,
which is just the obstruction class $\omega(E^{l'})$.
By Lemma \ref{obstruction=}, we have $\omega(E^{l'})=o(E^{l'})$.
We can see that $H^2(\mathrm{tr})(o(E^{l'}))=o(\det(E^{l'}))$.
Since the Picard scheme $\Pic_{X/k}$ is smooth over $k$,
we have $o(\det(E^{l'}))=0$.
So we have
\begin{align*}
 H^2(\mathrm{tr})(\omega(E^{\bullet})) &= H^2(\mathrm{tr})(\sigma(\varphi)) \\
 &= (-1)^{l'}H^2(\mathrm{tr})(\tau(\varphi))  \\
 &= (-1)^{l'}H^2(\mathrm{tr})(\omega(E^{l'}))  \\
 &=(-1)^{l'}H^2(\mathrm{tr})(o((E^{l'}))) \\
 &=(-1)^{l'}o(\det(E'))=0.
\end{align*}
Since the morphism
\[
 H^2(\mathrm{tr}):\Ext^2(E,I\otimes E) \longrightarrow
 H^2({\mathcal O}_X)\otimes I
\]
is isomorphic, we have
$\omega(E^{\bullet})=0$.
Thus $\splcpx\uet_{X/k}$ is smooth over $k$.
\end{proof}

The following theorem is essentially proved in
[\cite{H-L},II-10].
We give a proof again.

\begin{theorem}
 Let $X$ be an abelian or a projective K3 surface over
 an algebraically closed field $k$.
 Then $\mathrm{Splcpx}\uet_{X/k}$ has a symplectic structure,
 that is, there exists a closed $2$-form on
 $\mathrm{Splcpx}\uet_{X/k}$
 which is nondegenerate at every point.
\end{theorem}

\begin{proof}
Note that the tangent bundle $T_{\mathrm{Splcpx}\uet_{X/k}}$
on $\mathrm{Splcpx}\uet_{X/k}$ can be considered as the sheaf
on the small \'{e}tale site on $\mathrm{Splcpx}\uet_{X/k}$ defined by
\[
 U\mapsto
 \left\{ v\in \mathrm{Splcpx}\uet_{X/k}(U_{k[\epsilon]}) \left|
 \begin{array}{l}
 \text{the composite 
 $U\stackrel{i_0}\rightarrow U_{k[\epsilon]}
 \stackrel{v}\rightarrow \mathrm{Splcpx}\uet_{X/k}$} \\
 \mbox{is the structure morphism $U\to \mathrm{Splcpx}\uet_{X/k}$}
 \end{array}
 \right\}\right.,
\]
for any algebraic space $U$ \'{e}tale over $\mathrm{Splcpx}\uet_{X/k}$,
where $k[\epsilon]$ is the $k$-algebra generated by $\epsilon$
with $\epsilon^2=0$ and 
$U\stackrel{i_0}\rightarrow U_{k[\epsilon]}$
is the morphism induced by the ring homomorphism
\[
 k[\epsilon]\longrightarrow k; \quad \epsilon\mapsto 0.
\]

There is an \'{e}tale covering 
$\coprod_i U_i \to \mathrm{Splcpx}\uet_{X/k}$
such that $U_i\to \mathrm{Splcpx}\uet_{X/k}$ factors through
$\mathrm{Splcpx}_{X/k}$, that is, there is a universal family
$E^{\bullet}_{U_i}$ on each $X_{U_i}$.
Let $U$ be an affine scheme \'{e}tale over $\coprod_i U_i$
and $E^{\bullet}_U$ be the pull-back of the universal family.
Take any element $v\in T_{\mathrm{Splcpx}\uet_{X/k}}(U)$.
Then $U_{k[\epsilon]}\stackrel{v}\longrightarrow\mathrm{Splcpx}_{X/k}\uet$
factors through $\coprod_i U_i$,
since $\coprod_i U_i$ is \'etale over $\splcpx_{X/k}\uet$.
Let $E^{\bullet}_{U_{k[\epsilon]}}\in\mathrm{Splcpx}_{X/k}(U_{k[\epsilon]})$ 
be the pull-back of the universal family.
We can take a complex $\tilde{V}^{\bullet}$ of the form
$\tilde{V}^i=V_i\otimes_{{\mathcal O}_U}
{\mathcal O}_{X_{U_{k[\epsilon]}}}(-m_i)$
and a quasi-isomorphism $\tilde{V}^{\bullet}\to E^{\bullet}_{U_{k[\epsilon]}}$,
where $V_i$ is a locally free sheaf of finite rank on $U$,
$V_i=0$ for $i\gg 0$,
${\mathcal O}_X(1)$ is a fixed ample line bundle on $X$
and $\cdots\gg m_i\gg m_{i+1} \gg\cdots$.
Let $V^{\bullet}$ be the pull-back of $\tilde{V}^{\bullet}$ by
$X\times U\xrightarrow{\mathrm{id}_X\times i_0}
X\times U_{k[\epsilon]}$.
Then we obtain an element
\[
 [\{ d^i_{\tilde{V}^{\bullet}}-d^i_{V^{\bullet}}\otimes 1 \}]
 \in H^1(\Hom(\tilde{V}^{\bullet},\epsilon k[\epsilon]\otimes \tilde{V}^{\bullet}))
 \cong \Ext^1(E^{\bullet}_U,E^{\bullet}_U),
\]
which is independent of the choice of the representative
$\tilde{V}^{\bullet}$.
We can see that the mapping 
$v\to [\{ d^i_{\tilde{V}^{\bullet}}-d^i_{V^{\bullet}}\otimes 1 \}]$
defines an isomorphism
\[
 T_{\mathrm{Splcpx}_{X/k}\uet}(U)\stackrel{\sim}\longrightarrow
 H^0(U,\Ext^1_{X_U/U}(E^{\bullet}_U,E^{\bullet}_U)).
\]

For an affine scheme $U$ \'{e}tale over $\coprod_i U_i$,
there is a canonical pairing:
\[
 \begin{array}{ccc}
 \alpha_U:\Ext^1_{X_U/U}(E^{\bullet}_U,E^{\bullet}_U)\times
 \Ext^1_{X_U/U}(E^{\bullet}_U,E^{\bullet}_U) & \longrightarrow 
 & \Ext^2_{X_U/U}(E^{\bullet}_U,E^{\bullet}_U) \\
 (g,h) & \mapsto & g\circ h.
 \end{array}
\]
Note that there are canonical isomorphisms
\[
 \Ext^2_{X_U/U}(E^{\bullet}_U,E^{\bullet}_U)\stackrel{\sim}\longrightarrow
 \Ext^0_{X_U/U}(E^{\bullet}_U,E^{\bullet}_U)^{\vee}
 \stackrel{\sim}\longrightarrow {\mathcal O}_U.
\]
Then we can obtain a pairing
\[
 \alpha: T_{\mathrm{Splcpx}_{X/k}\uet}\times T_{\mathrm{Splcpx}_{X/k}\uet}
 \longrightarrow {\mathcal O}_{\mathrm{Splcpx}_{X/k}\uet}
\]
by patching $\alpha_U$.

Now we will see that $\alpha$ is skew-symmetric.
Take any $k$-valued point $p$ of
$\mathrm{Splcpx}_{X/k}\uet$.
$p$ corresponds to a complex
\[
 \cdots \longrightarrow V_i\otimes{\mathcal O}_X(-m_i)
 \stackrel{d^i_{V^{\bullet}}}\longrightarrow
 V_{i+1}\otimes{\mathcal O}_X(-m_{i+1})
 \stackrel{d^{i+1}_{V^{\bullet}}}\longrightarrow\cdots
 \longrightarrow V_l\otimes{\mathcal O}_X(-m_l)
 \longrightarrow 0\longrightarrow\cdots
\]
We denote this complex by $V^{\bullet}$.
Let us consider the restriction
\[
 \alpha(p):\Ext^1(V^{\bullet},V^{\bullet})\times \Ext^1(V^{\bullet},V^{\bullet})
 \longrightarrow \Ext^2(V^{\bullet},V^{\bullet})\cong k
\]
of the pairing $\alpha$.
Take any element 
$v=[\{v^i\}]\in H^1(\Hom^{\bullet}(V^{\bullet},V^{\bullet}))
\cong\Ext^1(V^{\bullet},V^{\bullet})$
and let $\tilde{V}^{\bullet}$ be a member of
$\splcpx_{X/k}(k[\epsilon])$
which corresponds to $v$.
$\tilde{V}^{\bullet}$ can be given by the complex
\[
 \cdots\longrightarrow
 V_i\otimes{\mathcal O}_X(-m_i)\otimes k[\epsilon]
 \xrightarrow{d^i_{V^{\bullet}}+\epsilon v^i}
 V_{i+1}\otimes{\mathcal O}_X(-m_{i+1})\otimes k[\epsilon]
 \longrightarrow\cdots
\]
Consider the surjection
$k[t]/(t^3)\to k[\epsilon];\: t\mapsto \epsilon$
and the extension
\[
 d^i_{V^{\bullet}}+tv^i:V_i\otimes{\mathcal O}_X(-m_i)\otimes k[t]/(t^3)
 \longrightarrow V_{i+1}\otimes{\mathcal O}_X(-m_{i+1})\otimes k[t]/(t^3)
\]
of the homomorphism
$d^i_{V^{\bullet}}+\epsilon v^i:V_i\otimes{\mathcal O}_X(-m_i)\otimes k[\epsilon]
\rightarrow V_{i+1}\otimes{\mathcal O}_X(-m_{i+1})\otimes k[\epsilon]$.
Then the obstruction class $\omega(\tilde{V}^{\bullet})$
for the lifting of $\tilde{V}^{\bullet}$ to a member of 
$\splcpx_{X/k}(k[t]/(t^3))$
with respect to the surjection
$k[t]/(t^3)\to k[\epsilon];\: t\mapsto \epsilon$
is given by
$[\{(d^{i+1}_{V^{\bullet}}+tv^{i+1})\circ(d^i_{V^{\bullet}}+tv^i)\}]
\in(t^2)\otimes\Ext^2(V^{\bullet},V^{\bullet})$.
However,
\[
 (d^{i+1}_{V^{\bullet}}+tv^{i+1})\circ (d^i_{V^{\bullet}}+tv^i)=
 d^{i+1}_{V^{\bullet}}\circ d^i_{V^{\bullet}}
 +t(d^{i+1}_{V^{\bullet}}\circ v^i+v^{i+1}\circ d^i_{V^{\bullet}})
 +t^2v^{i+1}\circ v^i=t^2v^{i+1}\circ v^i.
\]
Then
$\alpha(p)(v,v)=v\circ v=[\{v^{i+1}\circ v^i\}]=\omega(\tilde{V}^{\bullet})=0$
since $\mathrm{Splcpx}_{X/k}\uet$ is smooth over $k$.

Next we will see that $\alpha$ is nondegenerate.
The canonical isomorphism
\[
 {\bf R}{\mathcal Hom}(E^{\bullet}_U,E^{\bullet}_U)\stackrel{\sim}\longrightarrow
 {\bf R}{\mathcal Hom}^{\bullet}(E^{\bullet}_U,E^{\bullet}_U)^{\vee}
\]
induces the composite isomorphism by Grothendieck-Serre duality
\[
 \Ext^1(E^{\bullet}_U,E^{\bullet}_U)\stackrel{\sim}\longrightarrow
 \Ext^1({\bf R}{\mathcal Hom}^{\bullet}
 (E^{\bullet}_U,E^{\bullet}_U),{\mathcal O}_{X_U})
 \stackrel{\sim}\longrightarrow
 \Hom(\Ext^1(E^{\bullet}_U,E^{\bullet}_U),{\mathcal O}_U),
\]
which is just the homomorphism induced by $\alpha$.
Thus $\alpha$ is nondegenerate.

Finally we will show that $\alpha$ is $d$-closed.
For an affine scheme $U$ \'etale over $\coprod_i U_i $,
take $u,v,w\in T_{\splcpx_{X/k}^{\uet}}(U)$.
Let $E^{\bullet}\in\splcpx_{X/k}(U)$ be the pullback
of the universal family.
We may assume that there exists a complex
$V^{\bullet}$ of the form
$V^i=V_i\otimes{\mathcal O}_{X_U}(-m_i)$
such that $V^{\bullet}$ is quasi-isomorphic to $E^{\bullet}$
and that $V_i$ are vector spaces of finite dimension over $k$
and $m_i$ are integers.
Take $u\in T_U(U)$.
$u$ can be regarded as a derivation
${\mathcal O}_U\rightarrow{\mathcal O}_U$
over ${\mathcal O}_U$,
which is canonically extended to a derivation
\[
 D_u:V_i^{\vee}\otimes V_j\otimes{\mathcal O}_X(m_i-m_j)\otimes{\mathcal O}_U
 \longrightarrow
 V_i^{\vee}\otimes V_j\otimes{\mathcal O}_X(m_i-m_j)\otimes{\mathcal O}_U
\]
for $i\leq j$.
We have
$d_{V^{\bullet}}^{i+1}\circ D_u(d_{V^{\bullet}}^i)
+D_u(d_{V^{\bullet}}^{i+1})\circ d_{V^{\bullet}}^i=0$
for any $i$.
So we have
$[\{D_u(d_{V^{\bullet}}^i)\}]\in\Ext^1(V^{\bullet},V^{\bullet})$,
which corresponds to $u$ by the isomorphism
$T_U(U)\stackrel{\sim}\rightarrow\Ext^1(V^{\bullet},V^{\bullet})$.
Note that for $u,v\in T_U(U)$ we have
\[
 \alpha(u,v)=
 \left[\left\{D_u(d_{V^{\bullet}}^{i+1})\circ D_v(d_{V^{\bullet}}^i)\right\}\right]
 \in\Ext^2(V^{\bullet},V^{\bullet})\cong H^0(U,{\mathcal O}_U).
\]
For $u,v,w\in T_U(U)$, we have
\begin{align*}
 d\alpha(u,v,w)&=\left[\{
 D_u(\alpha(v,w))+D_v(\alpha(w,u))+D_w(\alpha(u,v))
 +\alpha(w,[u,v])+\alpha([u,w],v)+\alpha(u,[v,w]) \}\right] \\
 &=\left[\left\{ D_u(D_v(d_{V^{\bullet}}^{i+1})\circ D_w(d_{V^{\bullet}}^i))
 +D_v(D_w(d_{V^{\bullet}}^{i+1})\circ D_u(d_{V^{\bullet}}^i))
 +D_w(D_u(d_{V^{\bullet}}^{i+1})\circ D_v(d_{V^{\bullet}}^i)) 
 \right.\right. \\
 & \left.\left. \quad
 +D_w(d_{V^{\bullet}}^{i+1})\circ(D_uD_v-D_vD_u)(d_{V^{\bullet}}^i)
 +(D_uD_w-D_wD_u)(d_{V^{\bullet}}^{i+1})\circ D_v(d_{V^{\bullet}}^i)
 \right.\right. \\
 & \quad \left.\left.
 +D_u(d_{V^{\bullet}}^{i+1})\circ(D_vD_w-D_wD_v)(d_{V^{\bullet}}^i)
 \right\}\right] \\
 &=\left[\left\{
 D_uD_v(d_{V^{\bullet}}^{i+1})\circ D_w(d_{V^{\bullet}}^i)
 +D_v(d_{V^{\bullet}}^{i+1})\circ D_uD_w(d_{V^{\bullet}}^i)
 +D_vD_w(d_{V^{\bullet}}^{i+1})\circ D_u(d_{V^{\bullet}}^i)
 \right.\right. \\
 &\quad
 +D_w(d_{V^{\bullet}}^{i+1})\circ D_vD_u(d_{V^{\bullet}}^i)
 +D_wD_u(d_{V^{\bullet}}^{i+1})\circ D_v(d_{V^{\bullet}}^i)
 +D_u(d_{V^{\bullet}}^{i+1})\circ D_wD_v(d_{V^{\bullet}}^i) \\
 &\quad
 +D_w(d_{V^{\bullet}}^{i+1})\circ D_uD_v(d_{V^{\bullet}}^i)
 -D_w(d_{V^{\bullet}}^{i+1})\circ D_vD_u(d_{V^{\bullet}}^i)
 +D_uD_w(d_{V^{\bullet}}^{i+1})\circ D_v(d_{V^{\bullet}}^i) \\
 &\quad \left.\left.
 -D_wD_u(d_{V^{\bullet}}^{i+1})\circ D_v(d_{V^{\bullet}}^i)
 +D_u(d_{V^{\bullet}}^{i+1})\circ D_vD_w(d_{V^{\bullet}}^i)
 -D_u(d_{V^{\bullet}}^{i+1})\circ D_wD_v(d_{V^{\bullet}}^i)
 \right\}\right] \\
 &=\left[\left\{
 D_uD_v(d_{V^{\bullet}}^{i+1})\circ D_w(d_{V^{\bullet}}^i)
 +D_v(d_{V^{\bullet}}^{i+1})\circ D_uD_w(d_{V^{\bullet}}^i)
 +D_vD_w(d_{V^{\bullet}}^{i+1})\circ D_u(d_{V^{\bullet}}^i)
 \right.\right. \\
 & \quad \left.\left.
 +D_w(d_{V^{\bullet}}^{i+1})\circ D_uD_v(d_{V^{\bullet}}^i)
 +D_uD_w(d_{V^{\bullet}}^{i+1})\circ D_v(d_{V^{\bullet}}^i)
 +D_u(d_{V^{\bullet}}^{i+1})\circ D_vD_w(d_{V^{\bullet}}^i)
 \right\}\right] \\
 &=\left[
 \left\{ D_uD_vD_w(d_{V^{\bullet}}^{i+1}\circ d_{V^{\bullet}}^i)\right\}
 -\left\{ D_uD_vD_w(d_{V^{\bullet}}^{i+1})\circ d_{V^{\bullet}}^i
 +d_{V^{\bullet}}^{i+1}\circ D_uD_vD_w(d_{V^{\bullet}})
 \right\} \right] \\
 &=\left[ \{D_uD_vD_w(0)\}-d\left\{D_uD_vD_w(d_{V^{\bullet}}^i)\right\}\right] \\
 &=0.
\end{align*}
Here note that
\begin{align*}
 D_uD_vD_w(d_{V^{\bullet}}^{i+1}\circ d_{V^{\bullet}}^i)
 &=D_u(D_v(D_w(d_{V^{\bullet}}^{i+1}\circ d_{V^{\bullet}}^i)) )\\
 &=D_u(D_v(D_w(d_{V^{\bullet}}^{i+1})\circ d_{V^{\bullet}}^i
 +d_{V^{\bullet}}^{i+1}\circ D_w(d_{V^{\bullet}}^i))) \\
 &=D_u\left(D_vD_w(d_{V^{\bullet}}^{i+1})\circ d_{V^{\bullet}}^i
 +D_w(d_{V^{\bullet}}^{i+1})\circ D_v(d_{V^{\bullet}}^i) \right. \\
 &\quad \left.
 +D_v(d_{V^{\bullet}}^{i+1})\circ D_w(d_{V^{\bullet}}^i)
 +d_{V^{\bullet}}^{i+1}\circ D_vD_w(d_{V^{\bullet}}^i)\right) \\
 &= D_uD_vD_w(d_{V^{\bullet}}^{i+1})\circ d_{V^{\bullet}}^i
 +D_vD_w(d_{V^{\bullet}}^{i+1})\circ D_u(d_{V^{\bullet}}^i)
 +D_uD_w(d_{V^{\bullet}}^{i+1})\circ D_v(d_{V^{\bullet}}^i) \\
 & \quad +D_w(d_{V^{\bullet}}^{i+1})D_uD_v(d_{V^{\bullet}}^i)
 +D_uD_v(d_{V^{\bullet}}^{i+1})\circ D_w(d_{V^{\bullet}}^i)
 +D_v(d_{V^{\bullet}}^{i+1})\circ D_uD_w(d_{V^{\bullet}}^i) \\
 & \quad +D_u(d_{V^{\bullet}}^{i+1})\circ D_vD_w(d_{V^{\bullet}}^i)
 +d_{V^{\bullet}}^{i+1}\circ D_uD_vD_w(d_{V^{\bullet}}^i)
\end{align*}
So $\alpha$ is a closed $2$-form.
\end{proof}

\noindent
{\bf Acknowledgments.}
The author would like to thank Professors
Akira Ishii and K\={o}ta Yoshioka
for giving him the problem solved in this paper.
The author would also like to thank Professor
Fumiharu Kato for teaching him a fundamental
concepts of algebraic spaces.




%

\end{document}